\documentclass[12pt]{article}
\usepackage[all]{xy}
\usepackage{amsmath}
\usepackage{amsfonts}
\usepackage{amssymb}
\usepackage{theorem}

\title{Groups which are metrically weakly sofic with respect to word norms}
\author{Aleksander Ivanov 
}

\theoremstyle{plain}

\newtheorem{theorem}{Theorem}[section]
\newtheorem{proposition}[theorem]{Proposition}
\newtheorem{corollary}[theorem]{Corollary}
\newtheorem{lemma}[theorem]{Lemma}

\newtheorem{remark}[theorem]{Remark}
\newenvironment{proof}{\addvspace{8pt plus 2pt minus 2pt}\noindent\emph{Proof. }}
  { \begin{flushright}$\blacksquare$\par\addvspace{8pt plus 2pt minus
2pt}\end{flushright}}

\theoremstyle{definition}
\newtheorem{definition}[theorem]{Definition}

\theoremstyle{remark}

\begin{document} 
\topmargin = 12pt
\textheight = 630pt 
\footskip = 39pt 

\maketitle

\begin{quote}
{\bf Abstract} 
We consider metric versions of weak soficity, LEF and residual finiteness. 
The main results of the paper extend Glebsky and Rivera's characterization of weak soficity to the case of normally finitely generated groups with word metrics. 
Metric LEF and residual finiteness are also characterized in this class.  
We deduce that the free group $\mathsf{F}_2$ is not metrically weakly sofic with respect to its standard invariant word norm. 
\end{quote}


%

%
%

\section{Introduction} 

Let $G$ be a group and let $\Lambda$ be a closed convex subset of $[0,\infty)$ so that $0\in \Lambda$ (the case $\Lambda = [0,\infty )$ is possible). 
\begin{definition} \label{lf}
{\em A function $\ell:G\rightarrow \Lambda$ is called a {\em pseudo-length function (or pseudo-norm) } if for all $g$ and $h \in G$ \\ 
(i) $\ell(1) = 0 $; \\ 
(ii) $\ell(g)=\ell(g^{-1})$; \\
(iii) $\ell(gh)\le \ell(g) + \ell(h)$. \\ 
A {\em length function (or norm)} is a pseudo-length function which satisfies:  
\begin{quote} 
(i') for all $g\in G$ we have $\ell(g) = 0 $ if and only if $g=1$.   
\end{quote} 
} 
\end{definition} 
If $\ell$ is a $\Lambda$-(pseudo)-norm on $G$ then we say that $(G,\ell )$ is a {\em (pseudo) normed group}. 
The pseudo-norm $\ell$ is {\em invariant} if  $\ell(h^{-1} gh) = \ell(g)$ for all $g,h\in G$. 
In this case, it defines an invariant pseudo-metric by $d_{\ell}(g,h)=\ell(gh^{-1})$. 
Thus 
\[ 
\forall x,y,z (d_{\ell}(z \cdot x, z\cdot y) = d_\ell (x,y) = d_\ell (x\cdot z ,y\cdot z)). 
\] 
It becomes a metric if $\ell$ is a length function. 
On the other hand if $d$ is an invariant (pseudo)-metric on $G$ then the function $d(x,1)$ is an invariant (pseudo)-norm. 
In order to simplify notation it is convenient to work with norms instead of metrics. 
However there are situations when the word "metric" looks more appropriate than "norm".   
In particular we prefer "metric weak soficity" to "normed weak soficity".

In order to introduce the main topic of the paper we remind the reader the following definition. 

\begin{definition} \label{DefGR} 
{\em  (\cite{GR})
A group $G$ is called {\em weakly sofic }   
if there is a number $r\in \mathbb{R}$ such that for every 
finite subset $F\subseteq G$ and every $\varepsilon > 0$ there is 
a finite invariantly normed group $(C,\ell_C )$ and an injective map 
$\phi : F\rightarrow C$ such that 
\begin{itemize} 
\item for any triple $h,g,hg\in F$ we have  
$\ell_C ( \phi (hg ) (\phi (h) \phi (g))^{-1})<\varepsilon$, 
\item if $1 \in F$ then $\ell_C (\phi (1))< \varepsilon$, 
\item if $g \not= 1$ then $\ell_C (\phi (g)) > r$.
\end{itemize} 
}
\end{definition} 

Let us generalize this definition to normed groups. 
This will give the main object of this paper. 

\begin{definition} \label{DefMetrAppr} 
{\em We say that a normed group $(G,\ell )$ is {\em metrically weakly sofic }  if for every finite subset $F\subseteq G$ 
and every $\varepsilon >0$ there is a finite normed group $(C,\ell_C )$ and an injective map 
$\phi : F\rightarrow C$ such that 
\begin{itemize} 
\item for any triple $g,h,gh \in F$ we have  
$\ell_C (\phi (gh)^{-1} \phi (g) \phi(h)  ) < \varepsilon$, 
\item if  $1 \in F$  then 
$\ell_C (\phi(1)) < \varepsilon$,   
\item if $g\in F$ then $| \ell (g) - \ell_C (\phi (g))| <\varepsilon$. 
\end{itemize}  
}
\end{definition} 
This definition is folklore. 
It was explicitly considered in \cite{Iv} and implicitly in some other papers, for example in \cite{Doucha} and \cite{NST}. 
Note that the number $r$ from Definition \ref{DefGR} is not appropriate in Definition \ref{DefMetrAppr}, since the norm $\ell (g)$ determines (up to $\varepsilon$) the value $\ell_C (\phi (g))$.  
We also mention that a metric group $(G,d)$ with $d\le 1$ is  
metrically weakly sofic with respect to the norm corresponding to $d$ if and only if it isometrically embeds into a metric ultraproduct of finite metric groups.  

There are metric/normed groups for which the question of metric weak soficity looks very interesting. 
For example it would be interesting to know how it fits to bounded normal generation (with norms as in \cite{DT}), in particular in the case of Chevalley groups, see  \cite{Ta}.  
We also mention locally compact groups with property PL (i.e. unbounded length is proper), see \cite{dC}. 
Our original motivation is due to \cite{Doucha}, \cite{Iv}, \cite{NST}. 

In this paper we will also study the natural subclasses of metrically weakly sofic groups: metrically residually finite groups and metrically LEF groups. 
They will be defined in Section 2.3. 

\subsection{Word norms}

Let $G$ be a group generated by a symmetric set $S \subset G$. 
The {\em word norm} of an element $g \in G$ associated with $S$ is defined as follows:
\[ 
|g| = \mathsf{min}\{ k \in \mathbb{N} \, | \, g = s_1 \cdot \ldots \cdot s_k , \mbox{ where }s_i \in S\} . 
\] 
Let now $S\subseteq G$ be symmetric and $\overline{S}$ denote the the smallest conjugacy invariant subset of $G$ containing $S$. 
Assume that $G$ is generated by $\overline{S}$ but not necessarily by $S$. 
The following definition is taken from \cite{BGKM}. 
The {\em word norm} of an element $g \in G$ associated with $\overline{S}$ is defined by: 
\[ 
\parallel g\parallel = \mathsf{min}\{ k \in \mathbb{N}\, | \, g = s_1 \cdot \ldots \cdot s_k
, \mbox{ where }s_i \in \overline{S} \} .
\] 
This norm is conjugacy invariant. 
If $S$ is finite and $G$ is generated
by $\overline{S}$ then we say that $G$ is {\em normally finitely generated}. 

Finitely generated groups with word norms are especially visible in (geometric) group theory, see \cite{BGKM}, \cite{BIP},  \cite{Karl}, \cite{LS}, \cite{Sha} and \cite{Trost}.  
For example when $G$ is generated by finitely many conjugacy classes then finiteness of its diameter for the appropriate word metric  implies boundedness of its diameter with respect to any invariant metric, \cite{BGKM}.  
Thus metric weak soficity in the case of word norms is natural for investigations. 
This leads to an interesting subclass of {\em abstract} finitely generated weakly sofic groups: those ones which are metrically weakly sofic with respect to a word norm associated with some tuple of generators. 
It is an easy exercise that finitely generated abelian groups are metrically weakly sofic with respect to word norms arising from (Maltsev) bases. 
In fact they are metrically residually finite (See 2.3). 
There are other examples in \cite{IS} (which are metrically LEF).     
These results support the following problem: 
\begin{itemize} 
\item {\em Describe metrically weakly sofic groups for word norms in major group-theoretic classes. } 
\end{itemize} 
In order to approach this problem we give a characterization of the metric weak soficity in terms of profinite topology. 
The following statement is the main result of the paper. 

\bigskip 
\noindent 
{\bf Theorem 4.1}.
{\em Let $\mathsf{F}$ be a finitely generated free group and $G= \mathsf{F}/N$. 
Assume that $S\subseteq \mathsf{F}$ is a finite symmetric set representing elements of $G$ that define an invariant word norm, say $\parallel \cdot \parallel_N$, in $G$. 

The normed group $(G, \parallel \cdot \parallel_N )$ is metrically weakly sofic if and only if for any finite sequence $g_1$, $g_2$,..., $g_k\in G$  of $\parallel\cdot\parallel_N$-norm 1 and $h_1$, $h_2$,..., $h_{l}$ from $N$ (possibly with repetitions), the closure of every product  
$g_{i_1}^{\mathsf{F}} \cdot g_{i_2}^{\mathsf{F}} \ldots \cdot g_{i_m}^{\mathsf{F}}\cdot h_{1}^{\mathsf{F}} \cdot h_{2}^{\mathsf{F}} \ldots h_{l}^{\mathsf{F}}$ is a subset of $\{ w\in \mathsf{F} |\parallel wN \parallel_N \le m \} N$. 
}

\bigskip 
\noindent 
This theorem can be viewed as a metric version of the well-known theorem of Glebsky and Rivera characterizing weakly sofic groups \cite{GR}. 
Using this theorem we give some important information concerning the question formulated above: 
the free group $(\mathsf{F}_2 ,\parallel \cdot \parallel )$ where $\parallel \cdot \parallel$ corresponds to the free base, is not metrically weakly sofic (see Corollary \ref{Fnmws}. 
\footnote{Some other results concerning this question  can be found in \cite{IS}}
).

In Sections 3 and 5 we give theorems which characterize metric residual finiteness and metric LEF in terms of profinite topology. 
We also discuss a number of questions that arise in our approach.  

\section{Length functions on a group} 

This is a preliminary section. 
When we consider a normed group $(G,\ell )$, it is usually assumed that the norm is invariant. 
A free group with the standard length function $|\cdot |$ with respect to a free basis will be the only exception to this rule. 
Just in case in Section 2.1 we do not make the general invariantness assumption.  

\subsection{Standard properties of norms}

The following lemma is obvious. 

\begin{lemma} \label{corr1}  
Let $H<G$. 
For any (invariant) norm $\ell$ 
(metric $d$) on $G$ the restriction of $\ell$ (resp. $d$) to $H$ is an (invariant) norm (resp. metric). 
\end{lemma} 
Let $(G, \ell)$ be an (invariant) normed group, $g\in G$  and $r \in \Lambda$. 
The $r$-{\em ball of} $g$ is defined as follows: 
\[ 
B_r (g) = \{ h \in G \, | \, d_{\ell}(h,g) \le r \}. 
\] 
One can also consider the $(<r)$-{\em ball of} $g$: 
\[ 
B_{<r} (g) = \{ h \in G \, | \, d_{\ell}(h,g) < r \}. 
\] 
Then the family 
$\{ B_{<r}(g) \, | \, r\in (\Lambda \cap \mathbb{Q})\setminus \{ 0 \}, \, g\in G \}$ is a subbase of {\em the topology defined by} $\ell$. 
Thus, every normed group is also a topological group. 

The following statement is taken from \cite{ST} (Lemmas 2.1).  

\begin{lemma} \label{corr2} 
If $G$ is a group with an (invariant)  pseudo-norm $\ell$ and $H$ is a normal subgroup of $G$, then  
\[ 
\ell_{G/H} (gH )= \mathsf{inf} \{ \ell(gh): h\in H\} 
\] 
defines an (invariant) pseudo-norm on $G/H$. 
If $G$ is finite and $\ell$ is a norm, then $\ell_{G/H}$ is a norm too. 
\end{lemma}

\subsection{Metric homomorphisms and almost homomorphisms}

We now consider metric homomorphisms.

\begin{definition} \label{h} 
{\em Given two normed groups $(G_1 ,\ell_1 )$ and $(G_2 ,\ell_2 )$ we say that a map $\phi : G_1  \rightarrow G_2$ is a {\em metric homomorphism } if it is a group homomorphism and for every $h\in G_1$ we have 
$\ell_2 (\phi (h)) \le \ell_1 (h)$. 

If for every $h\in G_1$ the equality 
$\ell_2 (\phi (h)) = \ell_1 (h)$ 
holds, then we say that the homomorphism $\phi$ is {\em isometric}. 
}
\end{definition}

Note that in the case when norms $\ell_1$ and $\ell_2$ correspond to invariant metrics $d_1$ and $d_2$ then the condition $(\forall h)(\ell_2 (\phi (h)) \le \ell_1 (h))$ 
(resp. $\ell_2 (\phi (h)) = \ell_1 (h)$)
is equivalent to 
$(\forall g,h)(d_2 (\phi (g), \phi (h)) \le d_1 (g,h))$ (resp. $d_2 (\phi (g), \phi (h)) = d_1 (g,h)$). 

It is also worth noting that an isometric homomorphism of normed groups is necessary injective. 

We now introduce a metric version of almost homomorphisms from \cite{CSC}, p. 240. 

\begin{definition}  \label{mLEFm} 
{\em Given two groups $G_1$ and $G_2$ with (invariant) norms $\ell_1$ and $\ell_2$ respectively, let $D$ be a finite subset of $G_1$ and $Q$ be a finite subset of $\mathbb{Q} \cap \Lambda$ such that $0\in Q$. 
A map $\varphi :G_1 \rightarrow G_2$ is called a {\em $D$-$Q$-almost-homomorphism of normed groups $(G_1 ,\ell_1 )$ and $(G_2 , \ell_2 )$ } if $\varphi$ is injective on $D$ and   
\begin{itemize} 
\item every triple $h,g,hg\in D$ satisfies    
$\varphi (hg )=\varphi (h) \varphi (g)$, 
\item for any $g \in D$, any number $q\in Q$ and any symbol $\square \in \{ <\, , \, > \, , = \}$ we have   
\[ 
\ell_1 (g) \square q \Leftrightarrow \ell_2 (\varphi (g)) \square q.  
\]
\end{itemize}  
} 
\end{definition} 
This notion will be convenient in our arguments.

\subsection{Metric RF and metric LEF}

We emphasize that in general almost homomorphisms are not metric homomorphisms. 
In the following definition we amalgamate these notions. 
It gives interesting subclasses of metrically weakly sofic groups (when $\mathcal{C} = \mathcal{F}$, see below). 

\begin{definition} \label{mres}  
{\em 
Let $\mathcal{C}$ be a class of invariantly normed groups. 
A group $(G,\ell_G )$ with an invariant norm is called {\em (metrically) fully residually $\mathcal{C}$ }  if for every finite subset $D\subseteq G$ and every finite $Q \subset \Lambda \cap \mathbb{Q}$ with $0\in Q$ there is $(C,\ell_C )\in \mathcal{C}$ and a homomorphism $\varphi : G \rightarrow C$ 
(resp. metric homomorphism 
$\varphi: (G,\ell_G) \rightarrow (C,\ell_C )$) which is a $D$-$Q$-almost-homomorphism. 
}
\end{definition} 
Note that to be metrically fully residually $\mathcal{C}$ implies the property to be fully residually $\mathcal{C}$. 
When $\mathcal{C}$ is the class of all finite invariantly normed groups (denoted by $\mathcal{F}$), then we will use the the following notation for just defined notions. 
\begin{quote} 
RF for metric groups $=$ the property (the class) to be fully residually $\mathcal{F}$ as an invariantly normed  group. 

Metric RF $=$ the property (the class) to be metrically fully residually $\mathcal{F}$ for invariantly normed groups.  
\end{quote} 
It is easy to see that these classes of normed groups are subclasses of the class of metrically weakly sofic groups. 

We now define a wider subclass of metrically weakly sofic groups. 
In the beginning we remind the reader the following definition. 

\begin{definition} 
{\em (\cite{GV})
A group $G$ is called {\em LEF} 
if for every finite subset $D\subseteq G$ there is 
a finite group $C$ and an injective map 
$\phi : D\rightarrow C$ such that 
every triple $h,g,hg\in D$ satisfies 
$\phi (hg )=\phi (h) \phi (g)$. 
}
\end{definition} 
A generalization which will be studied below is as follows. 

\begin{definition}  \label{mLEF} 
{\em Let $\mathcal{C}$ be a class of normed groups (usually with invariant norms). 
A group $(G,\ell_G )$ with an invariant norm is called {\em  
locally embeddable into} $\mathcal{C}$ 
(shortly {\em metrically LE$\mathcal{C}$} ) if for every finite subset $D\subseteq G$ and every finite $Q \subset \Lambda \cap \mathbb{Q}$ with $0\in Q$ there is a normed group $(C,\ell_C )\in \mathcal{C}$ and a $D$-$Q$-almost-homomorphism  $\varphi : (G,\ell_G )\rightarrow (C,\ell_C )$. 
}
\end{definition} 

This definition is taken  from \cite{IS}. 
We will call the invariant normed group $(G,\ell_G )$ a {\em metrically LEF group} if it is metrically LE$\mathcal{C}$ where $\mathcal{C}$ is the class of all finite invariantly  normed groups $\mathcal{F}$. 
\begin{quote} 
Metric LEF $=$ metric LE$\mathcal{F}$. 
\end{quote}

\begin{remark} \label{norm-eq} 
{\em  It is worth noting that when in Definition \ref{mLEF} the values of $\ell_G$ or all possible $\ell_C$ belong to $\mathbb{Q}$ (in particular, when they are are integer-valued) the definition can be essentially simplified. 
The condition should be formulated as follows: for every finite $D\subset G$ there is a normed $(C,\ell_C )\in \mathcal{C}$ and an injective map 
$\phi : D\rightarrow C$ 
such that 
\begin{itemize} 
\item  every triple $h,g,hg\in D$ satisfies 
$\phi (hg )=\phi (h) \phi (g)$,  
\item for every $g \in D$ we have 
$\ell_G (g) = \ell_C (\phi (g))$. 
\end{itemize}  
}
\end{remark}

\subsection{Norms for RF and metric LEF} 

It is obvious that every locally finite group with an invariant length function is metrically LEF. 
On the other hand it is easy to find such a group which is not RF.  
The following statement gives examples of infinite finitely generated fully residually $\mathcal{F}$ groups. 
Note that in this statement $G$ is given together with a natural topology on it and the norm $\ell$ corresponds to this topology. 

\begin{proposition} \label{refi}
Let $G$ be a finitely generated residually finite group. 
Then there is an invariant norm $\ell$ such that $\ell$ defines the profinite topology on $G$ and $(G, \ell )$ is metrically RF. 
\end{proposition} 

\begin{proof} 
Let $N_1 > N_2 > \ldots$ be a descending sequence of normal subgroups of finite index with $\bigcap N_i = \{ 1 \}$. 
It is a basis of neighborhoods of the identity with respect to the profinite topology. 

Let $p$ be any prime number. 
Define 
\[  
\ell (g) = \mathsf{max} \{ \frac{1}{p^s} \, | \, g\not\in N_{s}\}  .  
\] 
It is easy to see that $\ell$ is an invariant norm which defines the profinite topology. 
In order to verify that $(G, \ell )$ is metrically fully residually finite take any finite $K\subset G$. 
Let $N_s$ be chosen so that $K\cap N_s \subseteq \{ 1 \}$. 
It is easy to see that for every $g\in K$ we have  
$\ell (g) = \ell_{G/N_s}(gN_s )$ where the right side norm is defined as in  Lemma \ref{corr2}.
\end{proof}

Let $(G, \ell )$ be an  invariantly normed group.  
A natural question arises in the situation when the set of values of $\ell$ is a proper subset of $\Lambda$, for example $\Lambda = \mathbb{R}^+$ and $\mathsf{Rng} (\ell) = \mathbb{N}$.  
{\em Does metric LEF of $(G, \ell)$ imply that $(G, \ell )$ is metrically LEF with respect to the subclass 
consisting of finite invariantly $\mathsf{Rng}(\ell )$-normed groups?} 
The following proposition concerns one of the simplest cases. 

\begin{proposition} \label{RvsN} 
Assume that 
$\Lambda_0 = \Lambda \cap\mathbb{N}$.  
We also assume that if $\mathsf{sup} (\Lambda)$ exists then it belongs to $\Lambda_0$. 

Let $\ell$ be an invariant norm on a group $G$ and 
$\mathsf{Rng} (\ell )\subseteq \Lambda_0$.  
Assume that $(G,\ell )$ is metrically LEF (resp. RF as a normed group). 

Then $(G,\ell )$ is metrically LEF (resp. RF as a normed group) with respect to the subclass 
consisting of finite invariantly normed groups with norms having the set of values in $\Lambda_0$. 
\end{proposition}  
 
\begin{proof}
Take a finite $K\subset G$, a finite $Q\subset \Lambda \cap \mathbb{Q}$ with $0\in Q$ and a $K$-$Q$-almost-homomorphism $\varphi$ from $(G,\ell )$ to some finite $(G_0 ,\ell_0 )$ with an invariant norm. 
It can happen that 
$\mathsf{Rng} (\ell_0 ) \not\subseteq \Lambda_0$. 
In order to correct this let us define a function $\ell'$ on $G_0$ as follows. 
If $g\in G_0$ and $\ell_0 (g) \in \Lambda_0$ put $\ell'(g) = \ell_0 (g)$. 
Note, that in this way we have $\ell (g) = \ell_0 (g)$ for all $g\in \varphi (K)$, see Remark \ref{norm-eq}. 

When $\ell_0 (g) \not\in \Lambda_0$  
let $n$ be the maximal natural number which is less than $\ell_0 (g)$. 
Define $\ell'(g') =n+1$ for all $g'\in G_0$ with $n < \ell_0 (g')\le n+1$.  
It is clear that this definition preserves (i) (resp. (i')), (ii) of Definition \ref{lf} and invariantness. 

In order to see property (iii) of Definition \ref{lf} let us apply induction. 
Let $m$ be the first number from $\Lambda_0$ such that $\ell_0$ takes values from $[0,m] \setminus \mathbb{N}$. 
Thus, these reals belong to the open interval $(m-1,m)$. 
Since all values of $\ell_0$ that do not accede to $m-1$, belong to $\mathbb{N}$, one easily sees that the (triangle) inequality (iii) holds for $\ell'$ and for all $g$, $h$ with $\ell_0 (gh) \le m$. 

If $g$ and $h$ satisfy 
$\ell_0 (gh) \in (m,m+1]$ then $\ell' (gh) = m+1$. 
If $\ell_0 (g) \not=\ell'(g)$ then $\ell'(g) = [\ell_0 (g)]+1$ (and so is true for $h$).  
In particular $\ell'(gh) \le \ell'(g) + \ell'(h)$.  
The rest of this induction is clear. 

The case of RF is similar. 
\end{proof}  
 
It seems to us that the method used in this proof can work in many other situations. 
In fact, we do not know any example in which the natural version of Proposition \ref{RvsN} does not hold. 
For example, it is easy to see that the statement of this proposition holds in the case when $\Lambda_0$ is dense in $\Lambda$.

\section{Approximations of finitely generated normed groups with respect to word norms} 

\subsection{Plan for Sections 3,4,5} 
 
Let $G$ be a group normally generated by a symmetric set $S \subset G$ and let $\overline{S}$ denote the the smallest conjugacy invariant subset of $G$ containing $S$. 
As in Introduction the word norm of an element $g \in G$ associated with $\overline{S}$ is defined by: 
\[ 
\parallel g\parallel = \mathsf{min}\{ k \in \mathbb{N}\, | \, g = s_1 \cdot \ldots \cdot s_k
, \mbox{ where }s_i \in \overline{S} \}  . 
\] 
It is integer valued and invariant. 
The following general problem will be studied in Sections 3,4,5.
 
\begin{itemize} 
\item Assume that $G = \langle \overline{S} \rangle$ where $S$ is a (finite) symmetric subset of $G$. 
When is $(G, \parallel \cdot \parallel )$ metrically weakly sofic (metrically RF or metrically LEF)? 
\end{itemize} 
Some of our results do not depend on the assumption of finiteness of $S$. 
This assumption will be added into formulations when it would be essential. 
 
Let $\mathsf{F}$ be a finitely generated free group with a fixed free basis and $N$ be a normal subgroup of $\mathsf{F}$. 
We will view $G$ as $G=\mathsf{F}/N$ and the set $S$ as above will be viewed as a subset of $\mathsf{F}$ (i.e. representatives of cosets).  
Let $|\cdot |$ be the word length on $\mathsf{F}$ with respect to the free basis of $\mathsf{F}$. 
We denote by $B^{\mathsf{F},|\cdot |}_n (g)$ the $n$-ball of $g$ with respect to this length. 
When $S$ is a subset $\mathsf{F}$ such that the set of cosets $SN$ normally generates $\mathsf{F}/N$,  
we will consider $G= \mathsf{F}/N$ together with an invariant word norm $\parallel \cdot \parallel_N$ defined with respect to 
$\overline{SN} = S^{\mathsf{F}}N$. 
Note that the subgroup $\langle S^{\mathsf{F}}\rangle$ of $\mathsf{F}$ also has the natural invariant word norm defined by $S^{\mathsf{F}}$. 
We denote it by $\parallel \cdot \parallel$.  
When $g \in \langle S^{\mathsf{F}}\rangle$ let $B^{\mathsf{F},\parallel\cdot\parallel}_j (g)$ denote the $j$-ball of $g$ in 
$\langle S^{\mathsf{F}}\rangle$
(thus in $\mathsf{F}$) with respect to 
$\parallel \cdot \parallel$. 
We put $B^{\mathsf{F},\parallel \cdot\parallel}_0 (g) = \{ g \}$. 
It is clear that 
\[ 
\parallel gN \parallel_N = \mathsf{min}\{ \parallel g'\parallel \, | \, g'\in gN \} . 
\]  
Our aim is to describe the properties from the problem above in terms of the profinite topology of $\mathsf{F}$. 
In Section 3 we will study metric RF, in Section 4 - metric weak soficity and in Section 5 - metric LEF. 
In Section 4 we will show that $\mathsf{F}_2$ is not metrically weakly sofic as a metric group of the following form.  
\begin{quote} 
Throughout the paper in the case when $G$ is a free group with a free base  
$x_0 , \ldots , x_i ,\ldots$, $i< \mathsf{n}$, where $\mathsf{n} \in \mathbb{N} \cup \{ \omega\}$, we consider the invariant word norm $\parallel \cdot \parallel$ which is defined by  
\[ 
S=\{ x^{\pm 1}_0 , \ldots , x^{\pm 1}_i ,\ldots \, | \,  i< \mathsf{n} \}. 
\]
\end{quote} 
In the remaining part of this introduction we give three useful remarks. 

\begin{remark} 
{\em (see Section 2 of \cite{GR}) 
Closed sets in the profinite topology can be characterized by the following separability property: a set $X \subseteq G$ is closed in the profinite topology if and only if for any $g \not\in X$ there exists
a homomorphism $\rho$ from $G$ to a finite group such that 
$\rho (g) \not\in \rho (X)$.
}
\end{remark} 

\begin{remark} 
{\em Assume that $G$ is a LEF group and is normally generated by a finite $S\subset G$. 
Let $\parallel \cdot \parallel$ be the corresponding invariant norm. 
Note that the following weak form of metric LEF holds. 

{\em For every finite subset $D\subseteq G$ there is a finite normed group $(C,\ell_C )$ and an injective map $\phi : D\cup S \rightarrow C$ 
such that} 
\begin{itemize} 
\item {\em every triple $h,g,hg\in D\cup S$ satisfies}  
$\phi (hg )=\phi (h) \phi (g)$, 
\item  $\ell_C$ {\em is an invariant word norm generated by $\phi(S)$ and for every} $g \in D$, 
\[ 
\ell_C (\phi (g)) \le \parallel g\parallel.  
\]
\end{itemize}  
Indeed, for every $g\in D$ find its presentation as $g= s^{v_1}_1 \cdot \ldots \cdot s^{v_k}_{k}$ where 
$k=\parallel g \parallel$, $s_i \in S$. 
Extend $D\cup S$ by all $v_j$ and all sub-products appearing in these expressions. 
Let $D_1$ be the corresponding set. 
Using LEF find a finite group $C$ and an appropriate embedding $\phi : D_1 \to C$ with 
$C=\langle \phi(D_1) \rangle$. 
Let $\ell_C$ be the word norm generated by $\overline{\phi (S)}$ in $C$.   
}
\end{remark} 

This remark also holds for {\em residual finiteness of normed groups} where $\phi$ to be an abstract homomorphism from $G$ to a finite group.  

In the third remark we again arrive at the the situation when $G= \langle \overline{S} \rangle$ with the corresponding word norm but now without assumption of finiteness of $S$. 
Note that the equalities $B_1 (1) =\overline{S} \cup\{ 1\}$ and $B_n (1)= (B_1 (1))^n$ always hold in such a group. 
We now notice that the converse also holds. 

\begin{lemma} 
Let $(G, \ell )$ be a group with an invariant integer valued norm such that $B_n (1) \subseteq (B_1 (1))^n$ for all $n\in \omega$.  
Then $\ell$ is a word norm with respect to $B_1 (1)$. 
\end{lemma} 

\begin{proof} 
If $g \in (B^H_1 (1)\cdot B^H_1 (1) )\setminus B^H_1 (1)$, then $\ell (g) \ge 2$.  
By the triangle inequality we have $\ell (g) = 2$. 
In particular $B^H_2 (1) = (B^H_1 (1))^2$. 
If $g \in (B^H_1 (1) )^3 \setminus (B^H_1 (1))^2$, then $\ell (g) \ge 3$.
By the triangle inequality we have $\ell (g) = 3$. 
The rest follows by easy induction. 
\end{proof}

\subsection{Metric RF} 

We now keep the assumptions of Section 3.1 and study the profinite topology of $\mathsf{F}$.

\begin{theorem} \label{mRF}
Let $G= \mathsf{F}/N$ and $\parallel \cdot \parallel_N$ be defined with respect to $S\subseteq \mathsf{F}$. 
Consider the following conditions. 
\begin{enumerate} 
\item The normed group $(G, \parallel \cdot \parallel_N )$ is RF as a normed group; 
\item The normed group $(G, \parallel \cdot \parallel_N )$ is metrically RF; 
\item For any $m\in \mathbb{N}$ the set 
$B^G_m (1) =\{ g \in G \, | \, \parallel g \parallel_N \le m \}$ 
is closed in the profinite topology of $G$;  
\item For any $m\in \mathbb{N}$ and any $w \not\in B^{\mathsf{F},\parallel \cdot\parallel}_m (1)N$ there is a homomorphism $\hat{\varphi}$ from $\mathsf{F}$ to a finite group such that $\hat{\varphi} (w) \not\in \hat{\varphi}(B^{\mathsf{F},\parallel \cdot \parallel}_m (1)N)$ and $\hat{\varphi} (N) = 1$. 
\end{enumerate} 
Then $2 \Rightarrow 3 \Rightarrow 4 \Rightarrow 1$, and 
assuming that $S$ is finite, $1 \Leftrightarrow 2$. 
\end{theorem}

\begin{proof}  
$2 \Rightarrow 3$.  
W.l.o.g. we may assume that there are elements $g\in G$ with 
$\parallel g \parallel_N \ge 2$. 
In order to prove that the set 
$B^G_m (1) = \{ x \in G \, | \, \parallel x \parallel_N \le m \}$ is closed, take any $g\not\in B^G_m (1)$. 
Choose a metric homomorphism $\varphi$ from $G$ to some finite $(H,\ell )$ such that 
$ \ell (\varphi (g)) = \parallel g\parallel_N$. 
If $\varphi (g) \in \varphi (B^G_m (1))$ then $\ell (\varphi (g)) \le m$ by the triangle inequality. 
Since $\parallel g \parallel_N = \ell (\varphi (g))$ we have a contradiction. 
  
$3 \Rightarrow 4$.                  
We preserve the notation of the previous paragraph. 
In particular $B^G_m (1)$ consists of all cosets from $B^{\mathsf{F},\parallel\cdot\parallel}_m (1) N$. 
If $w\not\in B^{\mathsf{F},\parallel \cdot \parallel}_m (1) N$, then $wN \not\in B^G_m (1)$. 
Choose a homomorphism $\varphi$ from $\mathsf{F}/N$ to some finite $H$ such that $\varphi (wN) \not\in \varphi(
B^G_m (1))$. 
Then the corresponding homomorphism $\hat{\varphi}: \mathsf{F} \rightarrow \mathsf{F}/N \to H$ gives 
$\hat{\varphi} (w) \not\in \hat{\varphi}(B^{\mathsf{F},\parallel \cdot \parallel}_m (1)N)$ and $\hat{\varphi} (N) =1$.  

$4 \Rightarrow 1$. 
Take any finite $F\subset \langle S^{\mathsf{F}}\rangle$
representing pairwise distinct elements modulo $N$. 
For every $f \in F$ find a presentation with the shortest $n_f$:   
\[ 
fN = g_{f,1} \cdot \ldots \cdot g_{f,n_f}N \, , \, \mbox{ where } 
g_{f, i} \in B^{\mathsf{F},\parallel\cdot\parallel}_1 (1) . 
\] 
Let $F_1$ be the set of all $g_{f,i}$ with $f\in F$ and $i\le n_f$. 
We may assume that $F_1 \subseteq F$.  
Then we associate $n_f =1$ to each $f\in F_1$. 

Take a finite $Q\subset \mathbb{N}$. 
Let $n= \mathsf{max} (|f|\, | \, f\in F)$. 
We may assume that $Q$ is an initial segment of natural numbers and contains $\mathsf{max}\{ \parallel fN\parallel_N \, | \, f\in B^{\mathsf{F},|\cdot |}_{2n} (1)\} +1$. 
Note that  $F \subseteq B^{\mathsf{F},|\cdot |}_{2n} (1)$, and  
$n_f = \parallel fN \parallel_N$ for each $f \in F$. 

By  the assumption of this part of the proof, 
for each $w\in  B^{\mathsf{F},| \cdot |}_{2n} (1) \setminus \{ 1\}$ 
there is a subgroup $H_w \triangleleft \, \mathsf{F}$ of finite index such that 
$N \subseteq H_w$ and for every $m\in Q$ with $w \not\in B^{\mathsf{F},\parallel \cdot\parallel}_{m} (1)N$ we have 
$wH_w \cap B^{\mathsf{F},\parallel \cdot\parallel}_{m} (1)N = \emptyset$. 
Now we find a subgroup of finite index of $\mathsf{F}$, say $H_F$, such that $N\subseteq H_F$ and for all $w\in B^{\mathsf{F},|\cdot |}_{2n} (1)\setminus \{ 1 \}$ and $m\in Q$ with 
$w\not\in B^{\mathsf{F},\parallel \cdot\parallel}_{m} (1)N$
we have 
$wH_F \cap B^{\mathsf{F},\parallel \cdot\parallel}_{m} (1)N = \emptyset$.  
In particular, taking $m=0$, we see $B^{\mathsf{F},| \cdot |}_{2n} (1)N \cap H_F \subseteq  N$. 
The latter obviously means that the cosets from 
$B^{\mathsf{F},|\cdot |}_n (1)N \subset \mathsf{F}/N$ have pairwise distinct representatives in $\mathsf{F}/H_F$. 

Let us define an invariant word norm on the finite group $\mathsf{F}/H_F$  taking the cosets from $F^{\mathsf{F}}_1 H_F$ 
as a generating set. 
It can happen that 
$\langle F^{\mathsf{F}}_1 H_F\rangle\not=\mathsf{F}/H_F$, i.e. the norm is defined only for elements of a proper subgroup of $\mathsf{F}/H_F$.
Then we just fix a sufficiently large natural number $\hat{m}$ (for example $\hat{m} > [\mathsf{F}:H_F]$) and take it as the norm for the elements of $\mathsf{F}/H_F \setminus \langle F^{\mathsf{F}}_1 H_F\rangle$. 
By the previous paragraph, for each $g\in F\setminus \{ 1 \}$ the norm of $gH_F\in \mathsf{F}/H_F$ defined in this way
coincides with $n_g = \parallel g\parallel_N$. 
We see that the finite group $\mathsf{F}/H_F$ has a norm that witnesses residual finiteness of $\mathsf{F}/N$ as a metric group for the set $FN$. 

$1 \Leftrightarrow 2$ for finite $S$. 
Assume 1. 
Let $v_1$, $v_2$,..., $v_m$ be a list of representatives of all conjugacy classes of $S^{\mathsf{F}}$. 
Thus 
\[ 
B^{\mathsf{F},\parallel\cdot\parallel}_1 (1)N =    
(v_{1}^{\mathsf{F}} \cup v_{2}^{\mathsf{F}} \ldots \cup v_{m}^{\mathsf{F}})N. 
\] 
Take any finite $F\subset \langle S^{\mathsf{F}}\rangle$
representing pairwise distinct elements modulo $N$. 
Choose a homomorphism $\varphi$ from $G$ onto some finite $(H,\ell )$ such that 
\[ 
\ell (\varphi (v_i N)) = \parallel v_i N\parallel_N (\mbox{ i.e.}=1), \, \,  
1 \le i \le m, \mbox{ and } 
\ell (\varphi (g)) = \parallel g\parallel_N \mbox{ for } g\in F.
\] 
Note that under the homomorphism $\varphi$ 
the norm does not increase, i.e. it is a metric homomorphism.  
Indeed, if $\parallel h \parallel_N = m$, then $h\in B^G_m (1)$, i.e. $\varphi (h) \in (\varphi (B^G_1 (1)))^m$. 
By the triangle inequality $\ell (\varphi (h)) \le m$.  
\end{proof}

\subsection{When metric RF does not depend on metric} 

Assume that $G$ is normally generated by finite $S_1$ and $S_2$. 
Let $\parallel \cdot \parallel_1$ and $\parallel \cdot\parallel_2$ be the corresponding word metrics of $G$. 
The following question looks very interesting. 
Does the property that $G$ is metrically LEF (metrically RF or metrically weakly sofic) with respect to one of the norms $\parallel \cdot \parallel_1$ and $\parallel \cdot \parallel_2$ imply the same property for another norm?   
We do not even know if these word norms generate the same topology in the following sense. 
Note that the topology defined by $\parallel \cdot \parallel_i$, $i\in \{ 1,2\}$, according to the recipe of  Section 2.1, is the discrete $\{ 0,1\}$-topology. 
On the other hand if we exclude from the sub-base of all $B_{<r}(g)$ the balls of size 1 then we obtain a coarser topology which is not necessarily discrete. 
We will say that this is the {\em corrected} topology defined by the norm. 
Using Theorem \ref{mRF} 
we see the following statement. 

\begin{corollary} 
Assume that 
$\parallel \cdot \parallel_1$   
and $\parallel \cdot \parallel_2$ 
generate the same corrected topology.  
Then $(G,\parallel \cdot \parallel_1 )$ is metrically RF if and only if so is $(G, \parallel \cdot \parallel_2 )$.  
\end{corollary}

\begin{remark} 
{\em When $G$ is normally generated by finite $S_1$ and $S_2$  
there is a natural number $K$ such that each element of $S_2$ is an $\overline{S}_1$-word of length $\le K$ and vice versa. 
The same property holds for $\overline{S}_1$ and $\overline{S}_2$. 
As a result we see that for any $g\in G$, $\parallel g\parallel_1 \le K \parallel g\parallel_2$ and $\parallel g\parallel_2 \le K \parallel g\parallel_1$, i.e. $\parallel \cdot \parallel_1$ and $\parallel \cdot \parallel_2$ are Lipschitz equivalent.  
}
\end{remark}

\section{Metric weak soficity}  
 
The following theorem characterizes metric weak soficity for word metrics. 
We preserve the notation of Section 3.1.  
 
\begin{theorem} \label{mws} 
Assume that the set $S\subseteq \mathsf{F}$ defining the word norm $\parallel \cdot \parallel_N$ is finite. 
The normed group $(\mathsf{F} /N, \parallel \cdot \parallel_N )$ is metrically weakly sofic if and only if for any finite sequence $g_1$, $g_2$,..., $g_k\in B^{\mathsf{F},\parallel\cdot\parallel}_1 (1)$ and $h_1$, $h_2$,..., $h_{l}$ from $N$ (possibly with repetitions) the closure of every product 
$g_{i_1}^{\mathsf{F}} \cdot g_{i_2}^{\mathsf{F}} \ldots \cdot g_{i_m}^{\mathsf{F}}\cdot h_{1}^{\mathsf{F}} \cdot h_{2}^{\mathsf{F}} \ldots h_{l}^{\mathsf{F}}$ is a subset of $B^{\mathsf{F},\parallel\cdot\parallel}_m (1) N$. 
\end{theorem} 

\begin{proof} 
Necessity.  
We want to show that the complement of $B^{\mathsf{F},\parallel\cdot\parallel}_m (1) N$ does not intersect the closure of the product 
$g_{i_1}^{\mathsf{F}} \cdot g_{i_2}^{\mathsf{F}} \ldots \cdot g_{i_m}^{\mathsf{F}}\cdot h_{1}^{\mathsf{F}} \cdot h_{2}^{\mathsf{F}} \ldots h_{l}^{\mathsf{F}}$. 
Assume $w \not\in B^{\mathsf{F},\parallel\cdot\parallel}_m (1) N$. 
Let $s =\parallel w N \parallel_N$. 
Thus $s>m$.  
Let $n$ be a natural number such that  \[ 
m (\mathsf{max} (| g_i | \, | i\le k )) + l (\mathsf{max}( |h_j | \, | j\le l )) + |w| <n. 
\]
Since $(\mathsf{F}/N, \parallel \cdot \parallel_N )$ is metrically weakly sofic, for any $\varepsilon >0$ there is a finite normed group $(H,\ell )$ and a map $\phi: {\mathsf{F}}/N \rightarrow H$ witnessing metric weak soficity for $\varepsilon$ and all words from $B^{\mathsf{F},|\cdot |}_{<n}(1)$. 
Take $\varepsilon < \frac{1}{4(mn +ln +m+l)}$. 

Let $\tilde{\phi}$ be the homomorphism $\mathsf{F} \to H$ defined by $\phi$ on the free basis of $\mathsf{F}$. 
It is easy to see that for any $v\in B^{\mathsf{F},|\cdot ]}_{<n}(1)$ we have $\ell ((\tilde{\phi}(v))^{-1} \phi (vN)) < 2n\varepsilon$ (we should take double $n\varepsilon$ since inverses of the elements of the base can also appear in $v$).  
It suffices to show that $\tilde{\phi}(w)$ does not belong to 
\[ 
\tilde{\phi} (g_{i_1}^{\mathsf{F}} \cdot g_{i_2}^{\mathsf{F}} \ldots \cdot g_{i_m}^{\mathsf{F}}\cdot h_{1}^{\mathsf{F}} \cdot h_{2}^{\mathsf{F}} \ldots h_{l}^{\mathsf{F}}). 
\] 
Assume the contrary, i.e. 
$\tilde{\phi}(w) = \tilde{\phi}(g^{u_1}_{i_1}\cdot g^{u_2}_{i_2} \ldots \cdot g^{u_m}_{i_m}\cdot h^{v_1}_{1} \cdot h^{v_2}_{2} \ldots \cdot h^{v_l}_{l})$. 
Since 
$|\ell (\phi (g_i N)) - 1 | <\varepsilon$ for $1 \le i \le k$, and 
$\ell (\phi (h_i N)) <\varepsilon$
for $1 \le i \le l$, 
we see that $\ell (\tilde{\phi} (g_i )) <1 + (2n+1)\varepsilon$
for $1 \le i \le k$, and
$\ell (\tilde{\phi} (h_i )) <(2n+1)\varepsilon$ for $1 \le i \le l$. 
By the triangle inequality for $\ell$ we see that 
\[ 
\ell (\tilde{\phi}(g^{u_1}_{i_1}\cdot g^{u_2}_{i_2} \ldots \cdot g^{u_m}_{i_m}\cdot h^{v_1}_{1} \cdot h^{v_2}_{2} \ldots \cdot h^{v_l}_{l}) <  m + (2n+1)m\varepsilon + (2n+1)l\varepsilon < m+1/2. 
\]  
On the other hand  
$s \le \ell (\phi (wN)))+\varepsilon  \le \ell (\tilde{\phi} (w)) + (2n+1)\varepsilon$. 
In particular 
$s <  m + 1/2 + (2n+1)\varepsilon <m +1$. 
This is a contradiction with the choice of $m$ and $s$.  
 
Sufficiency. 
Take a real $\varepsilon >0$ and a finite subset $F \subseteq \mathsf{F}$ representing pairwise distinct elements modulo $N$. 
For every $f \in F$ find a presentation with the shortest $n_f$:   
\[ 
fN = g_{f,1} \cdot \ldots \cdot g_{f,n_f}N \, , \, \mbox{ where } 
g_{f, i} \in B^{\mathsf{F},\parallel\cdot\parallel}_1 (1) . 
\] 
Let $F_1$ be the set of all $g_{f,i}$ with $f\in F$ and $i\le n_f$. 
We may assume that $F_1 \subseteq F$ and every conjugacy class from $S^{\mathsf{F}}$ is represented in $F_1$. 
Then we associate $n_f =1$ to each $f\in F_1$. 

Let $n$ be a natural number such that 
$B^{\mathsf{F},|\cdot |}_n (1)$ intersects each conjugacy class from 
$B^{\mathsf{F},\parallel\cdot \parallel}_1 (1)$ and 
\[ 
(|F| +2) \cdot \mathsf{max} (n_f \, | \, f\in F)\cdot  (\mathsf{max} (| g | \, | g\in F ))  <n. 
\]
Find a homomorphism to a finite group 
$\phi : \mathsf{F} \to H$ which is injective on the set of all elements of length $\le n$ and the following property holds: 
\begin{quote} 
$(\dagger )$ for every $w\in B^{\mathsf{F},|\cdot |}_n (1)$, 
for every $m\le \mathsf{max}(n_f \, | \, f\in F)$ such that 
$w \not\in B^{\mathsf{F},\parallel\cdot\parallel}_m (1)N$
and for every sequence $h_1 ,\ldots ,h_{l} \in B^{\mathsf{F},|\cdot |}_n (1)\cap N$ with $l\cdot \varepsilon \le m+1$,  
the element $\phi (w)$ does not belong 
to 
\[ 
(\phi (B^{\mathsf{F},\parallel\cdot \parallel}_1 (1)))^m\cdot \phi (h^{\mathsf{F}}_1 )\cdot \ldots \cdot \phi (h^{\mathsf{F}}_{l}). 
\] 
\end{quote} 
It is worth noting here that the set 
$(B^{\mathsf{F},\parallel\cdot \parallel}_1 (1))^m\cdot  h^{\mathsf{F}}_1 \cdot \ldots \cdot h^{\mathsf{F}}_{l}$ contains  
every set of the form  
\[ 
h_{1}^{\mathsf{F}} \cdot h_{2}^{\mathsf{F}}\cdot \ldots \cdot h_{l_1}^{\mathsf{F}} g_{i_1}^{\mathsf{F}} \cdot h_{l_1 +1}^{\mathsf{F}} \cdot h_{l_1 +2}^{\mathsf{F}} \ldots h_{l_2}^{\mathsf{F}} \cdot g_{i_2}^{\mathsf{F}} \cdot \ldots \cdot \cdot h_{\ell_m}^{\mathsf{F}}\cdot g_{i_m}^{\mathsf{F}}\cdot h_{l_{m+1} + 1}^{\mathsf{F}} \cdot h_{l_{m+1} +2}^{\mathsf{F}} \ldots \cdot h_{l}^{\mathsf{F}},  
\]
where $g_{i_j} \in B^{\mathsf{F},\parallel\cdot\parallel}_1 (1)$ and $l_1 < l_2 < \ldots < l_m < l$. 
In fact $(B^{\mathsf{F},\parallel\cdot \parallel}_1 (1))^m\cdot  h^{\mathsf{F}}_1 \cdot \ldots \cdot h^{\mathsf{F}}_{l}$ is the union of them.   

Take the invariant word norm on $H$ defined by all elements of  $\phi (B^{\mathsf{F},|\cdot |}_n (1))$ 
which belong to $\phi (B^{\mathsf{F},\parallel \cdot \parallel}_1 (1) \cup N)\setminus \{ 1\}$ and re-scale it for elements of $\phi(N)$ assigning the weight $\varepsilon$ to each representative of $\phi (B^{\mathsf{F},|\cdot |}_n (1)\cap N$). 

We view $\phi$ on cosets from $B^{\mathsf{F},|\cdot |}_n (1)N$ as a map which assigns $\phi (f)$ to $fN$ with $f\in B^{\mathsf{F},|\cdot |}_n (1)$. 
In order to verify the weak soficity condition for $\phi$, $\varepsilon$ and elements of $FN$ 
note that when $f_1 N\cdot f_2 N= f_3 N$ for $f_1 ,f_2,f_3 \in F$ we have that $\phi (f_1 )\phi (f_2) (\phi (f_3))^{-1}$ is equal to some $\phi (h)$ with $h\in N$, i.e an element of the norm $\varepsilon$.  
 
When $m\le \mathsf{max} (n_f \, | \, f\in F)$, $f\in F$ and $\parallel fN\parallel_N > m$, then the element $f$ does not belong to any  $(B^{\mathsf{F},\parallel\cdot \parallel}_1 (1))^m\cdot  h^{\mathsf{F}}_1 \cdot \ldots \cdot h^{\mathsf{F}}_{l}$ with 
$h_1 , \ldots , h_{l} \in N$.  
In particular by $(\dagger )$ (and the comment after it) the norm of $\phi (f)$ in $H$ is greater than $m$. 
Since $\parallel fN\parallel_N =n_f$ 
we see that the norm of $\phi (f)$ is not smaller than $n_f$. 
 
By the choice of $n$ the element $(g_{f,1} \cdot \ldots \cdot g_{f,n_f})^{-1}\cdot f$ belongs to 
$B^{\mathsf{F},|\cdot |}_n (1) \cap N$. 
Thus $\phi ((g_{f,1} \cdot \ldots \cdot g_{f,n_f})^{-1}\cdot f)\le \varepsilon$. 
In particular the norm of $\phi (f)$ in $H$ does not exceed $n_f + \varepsilon$. 
This shows that $\phi$, $FN$ and $\varepsilon$ satisfy the condition of the definition of metric weak soficity. 
\end{proof}

\bigskip

\begin{remark} \label{mws-nes} 
{\em In the proof of necessity of Theorem \ref{mws} we do not use the assumption of finiteness of $S$. 
Thus the following statement holds. \\  
{\em In the assumptions of the theorem admit that $S$ can be infinite and can meet infinitely many conjugacy classes in $\mathsf{F}$. 
Then if the normed group $(\mathsf{F}/N, \parallel \cdot \parallel_N )$ is metrically weakly sofic then 
for any finite sequence $g_1$, $g_2$,..., $g_k\in B^{\mathsf{F},\parallel\cdot\parallel}_1 (1)$ and $h_1$, $h_2$,..., $h_{l}$ from $N$ (possibly with repetitions) the closure of the product  
$g_{i_1}^{\mathsf{F}} \cdot g_{i_2}^{\mathsf{F}} \ldots \cdot g_{i_m}^{\mathsf{F}}\cdot h_{1}^{\mathsf{F}} \cdot h_{2}^{\mathsf{F}} \ldots h_{l}^{\mathsf{F}}$ is a subset of $B^{\mathsf{F},\parallel\cdot\parallel}_m (1) N$. 
}}
\end{remark}

The following statement is a part of the the main result of \cite{GR}. 

\begin{corollary} 
If the group $\mathsf{F} /N$ is weakly sofic then for any finite sequence $h_1$, $h_2$,..., $h_{\ell}$ from $N$ (possibly with repetitions) the closure of any product  
$h_{1}^{\mathsf{F}} \cdot h_{2}^{\mathsf{F}} \ldots h_{\ell}^{\mathsf{F}}$ is a subset of  $N$.
\end{corollary} 
\begin{proof} 
In order to see how it follows from Theorem \ref{mws} and Remark \ref{mws-nes} consider $\mathsf{F} /N$ together with the discrete $\{ 0,1\}$-metric. 
The latter one is the word metric corresponding to a set $S\subset \mathsf{F}$ which meets every non-trivial conjugacy class of $\mathsf{F}$. 
Now we may use Remark \ref{mws-nes}.  
\end{proof} 

Before the following statement we remind the reader that the free group $\mathsf{F}_2$ with a free base  
$x_0 , x_1$ is viewed as the word metric group with respect to $S=\{ x^{\pm 1}_0 , x^{\pm 1}_1 \}$, see Section 3.1. 

\begin{corollary} \label{Fnmws} 
The normed group $(\mathsf{F}_2 , \parallel \cdot \parallel )$ is not metrically weakly sofic. 
\end{corollary} 

\begin{proof} 
We will apply Theorem \ref{mws} in the situation when $N=\{ 1\}$. 
Since $B^{\mathsf{F}, \parallel \cdot \parallel}_1 (1)$ consists of four conjugacy classes, every  $B^{\mathsf{F}, \parallel \cdot \parallel}_m (1)$ is a union of finitely many products of the form 
$g_{1}^{\mathsf{F}} \cdot g_{2}^{\mathsf{F}} \ldots \cdot g_{k}^{\mathsf{F}}$ with $k\le m$ and $g_1, \ldots , g_k \in B^{\mathsf{F}, \parallel \cdot \parallel}_1 (1)$. 
Assuming that $(\mathsf{F}_2 , \parallel \cdot \parallel )$ is metrically weakly sofic we see by Theorem \ref{mws} that the ball $B^{\mathsf{F}, \parallel \cdot \parallel}_m (1)$ is a union of finitely many closed subsets. 
Thus $B^{\mathsf{F}, \parallel \cdot \parallel}_m (1)$ is also closed. 
This contradicts the following observation of D. Segal and J. Gismatullin (independently, see Remark 4 in \cite{NST}):  
\begin{quote} 
there is a natural number $n$ such that $(B^{\mathsf{F}, \parallel \cdot \parallel}_1 (1))^n$ is not closed in the profinite topology. 
\end{quote} 
The latter is a consequence of a deep theorem of Nikolov and Segal from \cite{NS} (see Theorem 1.2 there 
and explanations of A. Thom in \cite{thomover}). 
\end{proof} 

Although this statement looks to be unexpected, an example from \cite{NST} of a group which is not topologically weakly sofic together with the main construction from \cite{Doucha} suggest that some $\mathsf{F}_n$ is not metrically weakly sofic with respect to some (possibly graded) word metric.  
Indeed, any metric group (with $d\le 1$) 
embeds into a metric ultraproduct of finitely generated free groups with discrete bi-invariant metrics.  
This follows from the construction of Doucha of a universal 
separable group $\mathbb{G}$ equipped with a complete 
bi-invariant metric bounded by 1. 
According to \cite{Doucha} $\mathbb{G}$ is the completion 
of a Fra\"{i}ss\'{e} limit of free metric groups as above. 
As a result the following statement was already known since 2015:  
there is a finitely generated free group with a bi-invariant 
discrete metric which is not metrically weakly sofic.

\section{Metric LEF}

\subsection{Characterization of metric LEF} 
Our characterization of metric LEF clarifies the situations of Sections 3 and 4. 
We preserve the notation of Section  3.1. 

\begin{theorem} \label{mlef} 
Assume that the set $S\subseteq \mathsf{F}$ defining the word metric $\parallel \cdot \parallel_N$ is finite. 
The normed group $(\mathsf{F} /N, \parallel \cdot \parallel_N )$ is metrically LEF if and only if for any finite set $F \subset \mathsf{F}$ representing pairwise distinct elements modulo $N$, for any natural number $m$ and any $w \in \mathsf{F}\setminus B^{\mathsf{F},\parallel\cdot\parallel}_m (1) N$ there is a homomorphism $\varphi$ from $\mathsf{F}$ onto a finite group $H$ such that $\varphi (w) \not\in \varphi (B^{\mathsf{F},\parallel\cdot\parallel}_m (1) N)$ and the corresponding map $fN \to \varphi (f)$, $f\in F$, is a partial isomorphism 
$\mathsf{F}/N\to H$ defined on  $\{ fN | f\in F\}$. 
\end{theorem}

\begin{proof} 
Necessity. 
Take any finite $F$, a number $m$ and $w\in \mathsf{F}\setminus B^{\mathsf{F},\parallel \cdot \parallel}_m (1) N$ as in the formulation of the theorem. 
We may assume that $F$ contains $w$, a list of 
representatives of the conjugacy classes of $S^{\mathsf{F}}$,  
and the free generators of $\mathsf{F}$, and, furthermore, it is closed under the operation of taking subwords of a word. 

Let $n$ be a natural number such that  
\[ 
|F| (\mathsf{max} (|f| \, | \, f\in F)) + |S| (\mathsf{max} (| g| \, | g\in S )) +  |w| <n,  
\] 
Since $(\mathsf{F}/N, \parallel \cdot \parallel_N )$ is metrically LEF, there is a finite normed group $(H,\ell )$ and a map $\psi: {\mathsf{F}}/N \rightarrow H$ witnessing metric LEF for all words from $B^{\mathsf{F},|\cdot |}_{<n}(1)$. 
According to Remark \ref{norm-eq} and Proposition \ref{RvsN} we may assume that $\ell$ is integer-valued and the norm inequalities in the definition of LEF are equalities. 
In particular, 
\[ 
\ell (\psi (wN)) = \parallel w N\parallel_N \, \, \mbox{ and } \, \, \ell (\psi (g N)) = \parallel g N\parallel_N  \mbox{ for } g\in S \, \, (\mbox{ i.e. }=1).  
\] 
Let $\varphi$ be the homomorphism $\mathsf{F} \to H$ defined by $\psi$ on the free basis of $\mathsf{F}$. 
It is easy to see that the map $vN \to \varphi(v)$, $v\in B^{\mathsf{F},|\cdot |}_{<n}(1)$,  induces $\psi$ on $B^{\mathsf{F},|\cdot |}_{<n}(1)N$. 

If $\varphi (w) \in \varphi(B^{\mathsf{F},\parallel \cdot \parallel}_m (1)N)$ then 
the value $\ell (\varphi (w))$ is not greater than $m$ by the triangle inequality. 
This leads to a contradiction with  the equality 
$\parallel wN \parallel_N =\ell (\varphi (w))$.  

Sufficiency. 
To see that $(\mathsf{F} /N, \parallel \cdot \parallel_N )$ is metrically LEF 
take any finite $F\subset \langle S^{\mathsf{F}}\rangle$
representing some elements of $\mathsf{F}/N$. 
We may assume that $F$ is closed under taking subwords. 
Furthermore, for every $f \in F$ find a presentation with the shortest $n_f$:   
\[ 
fN = g_{f,1} \cdot \ldots \cdot g_{f,n_f}N \, , \, \mbox{ where } 
g_{f, i} \in B^{\mathsf{F},\parallel\cdot\parallel}_1 (1) . 
\] 
Let $F_1$ be the set of all $g_{f,i}$ with $f\in F$ and $i\le n_f$. 
We will also assume that $F_1 \subseteq F$ and that every conjugacy class from $S^{\mathsf{F}}$ is represented in $F_1$. 
Associate $n_f =1$ to each $f\in F_1$. 

Take a finite $Q\subset \mathbb{N}$. 
Let $n= \mathsf{max} (|f|\, | \, f\in F)$. 
We may assume that $Q$ is an initial segment of the set of natural numbers and contains $\mathsf{max}\{ \parallel fN\parallel_N \, | \, f\in B^{\mathsf{F},|\cdot |}_{2n} (1)\} +1$. 
Note that $n_f = \parallel fN\parallel_N$ for each $f \in F$. 
For every $w\in B^{\mathsf{F},| \cdot |}_{2n} (1) \setminus \{ 1\}$ 
there is a subgroup $H_w \triangleleft \, \mathsf{F}$ of finite index such that 
\begin{itemize} 
\item the map $fN \to fH_w$ is a partial isomorphism on $\{ fN | f\in B^{\mathsf{F},| \cdot |}_{2n} (1) \}$, 
\item 
for every $m\in Q$ with $w \not\in B^{\mathsf{F},\parallel \cdot\parallel}_{m} (1)N$ we have 
$wH_w \cap B^{\mathsf{F},\parallel \cdot\parallel}_{m} (1)N = \emptyset$. 
\end{itemize} 
Now it is easy to find a subgroup of finite index of $\mathsf{F}$, say $H_F$, such that these conditions are satisfied simultaneously for all $w\in B^{\mathsf{F},|\cdot |}_{2n} (1)\setminus \{ 1 \}$ and $m\in Q$ 
with $H_F$ instead of $H_w$. 

Let us define a word norm on the  finite group $\mathsf{F}/H_F $  taking the cosets from $F^{\mathsf{F}}_1 H_F$ (in fact $B^{\mathsf{F},\parallel \cdot\parallel}_1 (1)H_F$) as a generating set. 
By the previous paragraph for each $w\in F\setminus \{ 1 \}$ the norm of $wH_F \in \mathsf{F}/H_F$ defined in this way coincides with $n_w = \parallel wN\parallel_N$. 
We see that the finite group $\mathsf{F}/H_F $ has a norm  witnessing LEF for $FN$ in the normed group $(\mathsf{F}/N, \parallel \cdot \parallel_N )$. 
\end{proof}

\begin{corollary} \label{LEF-RF} 
Assume that the set $S\subseteq \mathsf{F}$ defining the word metric $\parallel \cdot \parallel_N$ is finite and that $N<\mathsf{F}$ is finitely generated as a normal subgroup.  
Then the condition that the normed group $(\mathsf{F} /N, \parallel \cdot \parallel_N )$ is
metrically LEF implies  that it is metrically RF.
\end{corollary} 

\begin{proof} 
Assume that $N$ is normally generated by words $w_1 , \ldots , w_k \in \mathsf{F}$. 
Let $m > \mathsf{max} (|w_1 | , \ldots , |w_k |)$. 
If a finite set $F \subset \mathsf{F}$ 
contains all words of length $\le m$ then $N$ is included into the kernel of every homomorphism $\varphi$ from $\mathsf{F}$ onto a finite group $H$ such that the corresponding map $fN \to \varphi (f)$, $f\in F$, is a partial isomorphism $\mathsf{F}/N\to H$ defined on  $\{ fN | f\in F\}$. 
In particular such a $\varphi$ 
induces a homomorphism $G =\mathsf{F}/N\to H$. 
We now see that the condition of Theorem \ref{mlef} implies that each 
$B^{G}_n (1)$ is closed in the profinite topology of $G$. 
By Theorem \ref{mRF} we obtain the conclusion of this corollary. 
\end{proof}

\subsection{Connections with metric weak soficity} 

As we already know metrically LEF groups are metrically weakly sofic. 
{\em When does the converse hold?} 
In order to answer this question we apply a kind of flexible stability, 
\cite{BB}. 
Note that a metric group $(G,d)$ with $d\le 1$ is  metrically weakly sofic if and only if it isometrically embeds into a metric ultraproduct of finite metric groups. 
In the following definition we give additional restrictions on such an embedding.  

\begin{definition} \label{LEFstable} 
{\em We say that a normed group $(G, \ell )$ is {\em LEF (flexibly) stable} if for every finite subset $F \subset G$ there is a function 
$f_F : [0,+\infty ) \rightarrow [0,+\infty )$ with $\lim_{x \to 0} f(x) = 0$ such that for all 
$\varepsilon$, $(C, \ell_C )$ and $\phi$ satisfying the conditions of Definition \ref{DefMetrAppr} there is a finite normed group $(D, \ell_D )$ and an injective map 
$\psi: F\rightarrow D$ such that 
any triple $h,g,hg\in F$ satisfies 
$\psi (hg )=\psi (h) \psi (g)$ and 
any $g\in F$ satisfies 
$| \ell (g) - \ell_D (\psi (g))| <f_F (\varepsilon )$. 
}
\end{definition}

The following statement is straightforward. 

\begin{proposition} 
A LEF stable metrically weakly sofic group is metrically LEF. 
\end{proposition}  

The fact that $(\mathsf{F}_2 , \parallel \cdot \parallel )$ is not metrically weakly sofic (see Corollary \ref{Fnmws}) 
motivates the question about LEF-stability of it. 
In the following proposition $\mathsf{n}$ can be infinite (recall notations of Section 3.1).

\begin{proposition} \label{f_g_s}  
The free group $(\mathsf{F_n}, \parallel \cdot \parallel )$ is LEF stable.   
\end{proposition} 

{\em Proof.} 
Let us fix a free base of $\mathsf{F_n}$ (which can be infinite).  
Let $F\subset \mathsf{F_n}$ be finite. 
Find $\bar{b}$, the minimal initial segment of the basis such that $F$ belongs to some ball $B_k (1)$ with respect to $\bar{b}$. 
Take the minimal $k$ with this property, i.e. $\mathsf{diam}(F) = 2k$ with respect to $\bar{b}$.  
We will prove that the function $f_F (x) = 3k \cdot x$ satisfies Definition \ref{LEFstable}. 
We may assume that $\bar{b} \subseteq F$ and $F$ is closed under taking subwords. 

Let $\varepsilon >0$ and 
$\phi : F \rightarrow C$ 
be a map into a finite metric group 
$(C , \ell_C )$ 
as in Definition \ref{DefMetrAppr}. 
We claim that the group $(C, \ell_C )$ can be taken as $(D,\ell_D )$ in Definition \ref{LEFstable} where $f_F$ is as above. 
Indeed, since $\mathsf{F_n}$ is free, the map 
$\bar{b} \rightarrow \phi(\bar{b})$ extends to a homomorphism 
$\mathsf{F_n} \rightarrow C$. 
We call it $\psi$. 
When $b_i$ and $b_j$ are taken from $\bar{b}$ then 
\[ 
\psi (b_i b_j ) = \psi (b_i ) \psi (b_j ) = \phi (b_i ) \phi (b_j ) 
\] 
and $d_C (\psi (b_i b_j ), \phi (b_i b_j )) \le \varepsilon$ (assuming that $b_i b_j \in F$). 
One can similarly show that $d_C (\phi (b_i^{-1} ), \psi (b_i^{-1})) \le \varepsilon$. 
Using invariantness of $d_C$ and properties of $\phi$ we obtain 
\[ 
d_C(\phi (b^{\epsilon_1}_i b^{\epsilon_2}_j ) ,\psi (b^{\epsilon_1}_i b^{\epsilon_2}_j ) )  = 
d_C (\phi ( b^{\epsilon_1}_i b^{\epsilon_2}_j ),\psi (b_i)^{\epsilon_1}  \psi (b_j )^{\epsilon_2}) 
\le 3\varepsilon ,
\] 
Applying this argument inductively we obtain that when $v\in F$ and $|v| = t$ then 
\[ 
d_C(\phi (v ) ,\psi (v ) )  < 2t \cdot \varepsilon ,
\] 
By  the triangle inequality: 
\[ 
|\ell_C (\phi (v )) - \ell_C (\psi (v ))| < 2t \cdot \varepsilon .
\] 
On the other hand 
$| \parallel v\parallel - \, \ell_C (\phi (v ) | < \varepsilon$.  
In particular, 
\[ 
| \parallel v \parallel - \, \ell_C (\psi (v) | < 3 t \cdot \varepsilon \le 3k \cdot \varepsilon.
\] 
$\Box$ 

\bigskip 

We finish this section by the following remark. 

\begin{remark} 
{\em It is proved in Lemma 5 of \cite{NST} that connected abelian Lie groups with the ``euclidean" length function are metrically weakly sofic. 
The proof given in that paper shows that 
\begin{itemize} 
\item for every $m$ and $n \in \omega$ the Lie group $\mathbb{R}^m \times (\mathbb{R}/\mathbb{Z})^n$ equipped with the ``euclidean" metric is metrically LEF. 
\end{itemize}
}
\end{remark}

\subsection{Further connections with metric RF}

Proposition \ref{f_g_s} and Corollary \ref{LEF-RF} suggest that  Corollary \ref{Fnmws} can be deduced by proving that $( \mathsf{F}_2, \parallel \cdot \parallel )$  is not RF. 
In this section we discuss this way of argumentation in a slightly more general situation. 

Assume that $w(z_1,z_2,\ldots, z_m )$ is a group word and $V_w$ be the variety of groups defined by this word as an identity. 
By $\mathsf{F}^w_n$ we denote the free $n$-generated group of $V_w$. 
We will say that the word $w(\bar{z})$ is {\em quasi-linear} if for any group $H$ and its generating set $A$ the group $H$ belongs to $V_w$ if and only if $w(\bar{a})=1$ for each $n$-tuple $\bar{a}$ from $A$.  
Note that the trivial word is quasi-linear. 
Any multi-linear word (or an outer commutator identity) is quasi-linear too. 
In particular the variey of $\ell$-step nilpotent groups has the form $V_w$ for a quasi-linear $w(\bar{z})$. 

When $\mathsf{F}^w_{\mathsf{n}}$ is a free group of a group variety with a free base  
$x_0 , \ldots , x_i ,\ldots$, $i< \mathsf{n}$, where $\mathsf{n} \in \mathbb{N} \cup \{ \omega\}$, we preserve the notation above, i.e. we will assume that the set 
\[ 
S=\{ x^{\pm 1}_0 , \ldots , x^{\pm 1}_i ,\ldots \, | \,  i< \mathsf{n} \}. 
\]
generates the word metric $\parallel \cdot \parallel$. 

\begin{proposition} \label{f_g_1} 
Let $w(\bar{z})$ be a quasi-linear word and $\mathcal{C}$ be a class of normed groups such that their norms are integer valued and $\mathcal{C}$ is closed under subgroups. 

Then the free group $(\mathsf{F}^w_{\mathsf{n}} , \parallel \cdot \parallel )$ 
(where $\mathsf{n}$ is greater than $1$) is metrically LE$\mathcal{C}$ if and only if it is metrically fully residually $\mathcal{C}$.    
\end{proposition} 

\begin{proof} 
Sufficiency is obvious. 
Let us prove necessity. 
Let $D$ be a finite subset of $(\mathsf{F}^w_{\mathsf{n}} , \parallel \cdot \parallel )$ and $Q$ be a finite subset of $\mathbb{N}$ with $0\in Q$. 
Let $\varphi :(\mathsf{F}^w_{\mathsf{n}} , \parallel \cdot \parallel )\rightarrow (H, \ell ) \in \mathcal{C}$ be as in  Definition  \ref{mLEF}. 
We may assume that there is $D_0 \subseteq D$ such that $S \cap D = D_0$, $D\subseteq \langle D_0 \rangle$ and for each sub-term $t(\bar{z})$ of $w(\bar{z})$ all values of the form $t(\bar{d})$ with $\bar{d}$ from $D_0$ belong to $D$.   
By quasi-linearity of $w(\bar{z})$ the set $\varphi (D  )$ generates a subgroup of $H$ which belongs to $V_w$. 
In particular, we can simultaneously extend $\varphi$ and $S \setminus D_0 \to 1$ to a homomorphism into $H$, say $\hat{\varphi}$.  
Taking a subgroup if necessary we assume that $\hat{\varphi}$ is surjective. 

Since $S \setminus \mathsf{Ker} \hat{\varphi}\subseteq D$, $\hat{\varphi} (S)\subseteq B^H_1 (1)$. 
The latter implies $\hat{\varphi} (\overline{S}) \subseteq B^H_1 (1)$ by invariantness of $\ell$. 
Now 
\[ 
\hat{\varphi} (B^{\mathsf{F}^w_{\mathsf{n}}}_m (1)) = \hat{\varphi} (\overline{S}^m) = 
\hat{\varphi} (\overline{S})^m \subseteq B^H_m (1). 
\] 
Indeed, the latter inclusion follows from the triangle inequality.  
We see that $\ell (\hat{\varphi}(g)) \le \parallel g \parallel$ for any $g\in \mathsf{F}^w_{\mathsf{n}}$. 
The rest is clear. 
\end{proof}  

We now repeat the argument of Corollary \ref{Fnmws}.  
This gives a slightly weaker statement, but in a different way. 

\begin{corollary} \label{F_2not}
The free group $(\mathsf{F}_2, \parallel \cdot \parallel )$  is not metrically LEF. 
\end{corollary} 

\begin{proof} 
Let $S= \{ x^{\pm 1}_1 , x^{\pm 1}_2 \}$. 
Note that $(\mathsf{F}_2, \parallel \cdot \parallel )$  is not metrically LEF. 
Indeed, use Proposition \ref{f_g_1},  Theorem \ref{mRF} and an observation of D. Segal and J. Gismatullin 
there is a natural number $n$ such that $\overline{S}^n$ is not closed in the profinite topology. 
\end{proof}

Applying Proposition \ref{f_g_s} and  Corollary \ref{F_2not} we now obtain the following corollary.  

\begin{corollary} 
The metric group $(\mathsf{F}_2 ,\parallel \cdot \parallel )$ is not metrically weakly sofic.  
\end{corollary}

\begin{remark}
{\em The following question looks very interesting. 
For $n>1$ consider $SL(n,\mathbb{Z})$ together with the word metric $\parallel \cdot \parallel$ associated with respect to the standard set of generators (i.e. $\{ 0,1\}$-transvections).  
{\em Is this group metrically LEF/weakly sofic? }
}
\end{remark}

{\bf Acknowledgment.}
The author is grateful to Oleg Bogopolski for the reference \cite{thomover}.


Institute of Computer Science, University of Opole, 

ul. Oleska 48, 45 - 052 Opole, Poland 

aleksander.iwanow@uni.opole.pl 

and


Department of Applied Mathematics, Silesian University of Technology, 

ul. Kaszubska 23, 44 -- 101 Gliwice, Poland 

Aleksander.Iwanow@polsl.pl

\end{document}